\documentclass[11pt]{amsart}
\usepackage{eurosym}
\usepackage{amssymb}
\usepackage{amsfonts}
\usepackage{amsmath}
\usepackage{xcolor}
\usepackage{bbm}

\newtheorem{theorem}{Theorem}[section]
\newtheorem{proposition}{Proposition}[section]

\theoremstyle{definition}

\theoremstyle{remark}
\newtheorem{remark}[theorem]{Remark}
\numberwithin{equation}{section}


\begin{document}
\title[Multiple solutions to Gierer-Meinhardt type systems]{Three solutions
with precise sign properties for Gierer-Meinhardt type system}

\begin{abstract}
We establish the existence of three solutions for sign-coupled
Gierer-Meinhardt type system with Neumann boundary conditions. Two solutions
are of opposite constant-sign while the third solution is nodal with
synchronous sign components. The approach combines sub-supersolutions method
and Leray-Schauder topological degree involving perturbation argument.
\end{abstract}

\author{Abdelkrim Moussaoui}
\address{Department of Physico-Chemical Biology, Faculty of Natural and Life
Sciences\\
and Applied Math. Laboratory, Faculty of Exacte Sciences \\
A. Mira Bejaia University, Algeria}
\email{abdelkrim.moussaoui@univ-bejaia.dz}
\maketitle

\section{Introduction}

\label{S1}

Let $\Omega $ is a bounded domain in $%
\mathbb{R}
^{N}$ $\left( N\geq 2\right) $ with a smooth boundary $\partial \Omega $. We
consider the following system of semilinear elliptic equations%
\begin{equation*}
(\mathrm{P})\qquad \left\{ 
\begin{array}{ll}
\Delta u-u+f_{1}(v)(\frac{|u|^{\alpha _{1}}}{|v|^{\beta _{1}}}+\rho )=0 & 
\text{in }\Omega , \\ 
\Delta v-v+f_{2}(u)\frac{|u|^{\alpha _{2}}}{|v|^{\beta _{2}}}=0 & \text{in }%
\Omega , \\ 
\frac{\partial u}{\partial \eta }=\frac{\partial v}{\partial \eta }=0\text{
\ on }\partial \Omega , & 
\end{array}%
\right.
\end{equation*}%
where $\Delta $ stands for the Laplace differential operator, $\eta $
denotes the unit outer normal to $\partial \Omega $ and $\rho >0$ is a
parameter.\ The exponents\ $\alpha _{i}\in (0,1)$ and $0\leq \beta _{i}<1$\ (%
$i=1,2$) satisfy the following condition 
\begin{equation}
\max \{\alpha _{1}+2\beta _{1},\text{ }\alpha _{2}+\frac{\beta _{2}}{2}\}<1,
\label{alpha}
\end{equation}%
while the functions $f_{i}\in L^{\infty }(\Omega )$ defined by $%
f_{i}(s):=f_{i}(sgn(s)),$ for all $s\in 
\mathbb{R}
,$ satisfy%
\begin{equation*}
sgn(f_{i}(s))=\left\{ 
\begin{array}{ll}
1 & \text{for }s\geq 0, \\ 
-1 & \text{for }s<0,%
\end{array}%
\right. ,\text{ for }i=1,2,
\end{equation*}%
where $sgn(\cdot )$ denotes the sign function. Functions $f_{1}$ and $f_{2}$
suggest that system $(\mathrm{P})$ is sign-coupled.\ This is expressed by
the fact that the first (resp. second) equation of $(\mathrm{P})$ depends on
the sign of the second (resp. first) component $v$. When only positive
solutions $(u,v)$ are considered, $f_{1}(v)\equiv f_{2}(u)\equiv 1$ and
therefore, system $(\mathrm{P})$ is reduced to%
\begin{equation*}
\left\{ 
\begin{array}{ll}
\Delta u-u+\frac{u^{\alpha _{1}}}{v^{\beta _{1}}}+\rho =0 & \text{in }\Omega
, \\ 
\Delta v-v+\frac{u^{\alpha _{2}}}{v^{\beta _{2}}}=0 & \text{in }\Omega ,%
\end{array}%
\right.
\end{equation*}%
which has been the focus of particular attention in the contexts of Neumann
and Dirichlet boundary conditions (see, e.g., \cite{1, 2, M2, LPS}).

By a solution of problem $(\mathrm{P})$ we mean $(u,v)\in \mathcal{H}%
^{1}(\Omega )\times \mathcal{H}^{1}(\Omega )$ such that 
\begin{eqnarray*}
\int_{\Omega }(\nabla u\nabla \varphi _{1}+u\varphi _{1})\ \mathrm{d}x
&=&\int_{\Omega }f_{1}(v)(\frac{|u|^{\alpha _{1}}}{|v|^{\beta _{1}}}+\rho
)\varphi _{1}\ \mathrm{d}x, \\
\int_{\Omega }(\nabla v\nabla \varphi _{2}+v\varphi _{2})\ \mathrm{d}x
&=&\int_{\Omega }f_{2}(u)\frac{|u|^{\alpha _{2}}}{|v|^{\beta _{2}}}\varphi
_{2}\ \mathrm{d}x,
\end{eqnarray*}%
for all $\varphi _{1},\varphi _{2}\in \mathcal{H}^{1}(\Omega ),$ provided
the integrals in the right-hand side of the above identities exist.

System $(\mathrm{P})$ is the elliptic counterpart of Gierer-Meinhardt model 
\cite{GMsyst}, proposed in 1972, which is a typical example of a
reaction-diffusion system that has been extensively studied in recent years.
The general model proposed by Gierer and Meinhardt may be written as%
\begin{equation*}
(\mathrm{GM})\qquad \left\{ 
\begin{array}{ll}
u_{t}=d_{1}\Delta u-\hat{d}_{1}u+c\rho \frac{u^{\alpha _{1}}}{v^{\beta _{1}}}%
+\rho \text{ } & \text{in }\Omega \times \left[ 0,T\right] , \\ 
v_{t}=d_{2}\Delta v-\hat{d}_{2}v+c^{\prime }\rho ^{\prime }\frac{u^{\alpha
_{2}}}{v^{\beta _{2}}} & \text{in }\Omega \times \left[ 0,T\right] ,%
\end{array}%
\right.
\end{equation*}%
subject to Neumann boundary conditions. The constants $\hat{d}_{1},\hat{d}%
_{2},c,c^{\prime }$ and $\rho $ are positive, $d_{1},d_{2}$ are diffusion
coefficients with $d_{1}\ll d_{2},$ the exponents $\alpha _{i},\beta
_{i}\geq 0$ satisfy the relation $\beta _{1}\alpha _{2}>\left( \alpha
_{1}-1\right) \left( \beta _{2}+1\right) $. System $(\mathrm{GM})$ describes
the interaction between activator $u(t,x)$ and inhibitor $v(t,x)$ in diverse
biological systems, with a particular emphasis on those pertaining to cell
biology and physiology.

The elliptic system $(\mathrm{GM})$ have attracted significant interest,
resulting in a substantial number of research papers. When $d_{2}$
approaches infinity, the existence, stability, and dynamics of spike
positive solutions have been investigated in \cite{GG, GWW, Ni-Takagi,
Ni-Takagi2, W1}. Conversely, when $d_{2}$ is bounded ($d_{2}<+\infty $), the
focus shifts to the analyses presented in \cite{JN,LPS, NI-Takagi-Yanagida,
WW, Wei2}. Extending the spatial domain to the whole space $\Omega =%
\mathbb{R}
^{N}$, \cite{MKT} (for $N\geq 3$), \cite{DKC, DKW} (for $N=1,2$), and \cite%
{KR, KWY} (for $N=3$) have addressed the existence, uniqueness, and
structural features of positive solutions for Gierer-Meinhardt type systems $%
\left( \mathrm{P}\right) $. In the specific case where $d_{1}$ and $d_{2}$
both equal $1,$ the Neumann elliptic system $(\mathrm{GM})$ is reduced to $%
\left( \mathrm{P}\right) $ . In this context, when $\rho \equiv 0,$ system $%
\left( \mathrm{P}\right) $ has been recently studied in \cite{M2, M4},
showing the existence of three distinct solutions. In \cite{M2}, the
obtained solutions are all positive while in \cite{M4}, where $\left( 
\mathrm{P}\right) $ is subjected to Dirichlet boundary conditions, it has
been established that one of the solutions is nodal and located between two
opposite constant-sign solutions. Recall from \cite{M3, M4} that a solution
for system $\left( \mathrm{P}\right) $ whose components at least are not of
the same constant-sign is nodal. We mention that, unlike to what has been
stated in \cite{M2}, and in the line with what has been established in \cite%
{MedjMous}, there can be no solutions to Neumann problem $\left( \mathrm{P}%
\right) $ with zero trace condition on $\partial \Omega $. Therefore, only
two of the three positive solutions obtained for the Neumann-type system $%
\left( \mathrm{P}\right) $ in \cite{M2} should be retained.

Our main purpose is to establish the existence of three distinct solutions
for Gierer-Meinhardt system $\left( \mathrm{P}\right) $ with a precise sign
information: two of them are of opposite constant-sign, while the third is
nodal with synchronous sign-changing components. The main result is
formulated as follows.

\begin{theorem}
\label{T}Under assumption (\ref{alpha}), problem $(\mathrm{P})$ admits at
least two opposite constant-sign solutions $(u_{+},v_{+})\in int\mathcal{C}%
_{+}^{1}(\overline{\Omega }))\times int\mathcal{C}_{+}^{1}(\overline{\Omega }%
),$ $(u_{-},v_{-})\in -int\mathcal{C}_{+}^{1}(\overline{\Omega })\times -int%
\mathcal{C}_{+}^{1}(\overline{\Omega }).$ If $\beta _{1}=0$, $(\mathrm{P})$
has a third nodal solution $(u^{\ast },v^{\ast })\in \mathcal{H}^{1}(\Omega
)\times \mathcal{H}^{1}(\Omega )$ satisfying $u^{\ast }v^{\ast }>0$ a.e. in $%
\Omega $.
\end{theorem}

The proof combines sub-supersolutions techniques and topological degree
theory. It falls into two parts, each corresponding to the statements of
Theorems \ref{T1} and \ref{T2}.\ The existence of opposite constant sign
solutions $(u_{+},v_{+})$ and $(u_{-},v_{-})$ to system $(\mathrm{P})$ is
stated in Theorem \ref{T1}. They are located in positive and negative
rectangles formed by two opposite constant sign sub-supersolutions pairs.
The latter are constructed by a choice of suitable functions with an
adjustment of adequate constants. Furthermore, for any positive solution $%
(u_{+},v_{+})$ and negative solution $(u_{-},v_{-})$ enclosed within the
rectangle formed by the opposite supersolutions, we show that the components 
$u_{+}$ and $u_{-}$ are invariably greater and less than their corresponding
positive and negative subsolutions. This strongly indicates that any
solution is nodal if its first component is positive and less than the
positive subsolutions or negative and greater than the negative subsolution.
This point is crucial to show the existence of a nodal solution $(u^{\ast
},v^{\ast })$ provided by Theorem \ref{T2}. Using suitable truncation
arguments and topological degree theory, we provide a third solution $%
(u^{\ast },v^{\ast })$ to problem $(\mathrm{P})$ that lies between the
previously specified positive and negative rectangles. The aforementioned
conclusion is thus the consequence of the sign-coupling of system $(\mathrm{P%
})$. This further shows that the components $u^{\ast }$ and $v^{\ast }$ are
synchnous sign-changing. We note that a control near the singularity of all
the terms involved in problem $(\mathrm{P})$ represents a significant part
of the argument. This necessarily involves the reconfiguration of the
competitive system $(\mathrm{P})$ to a cooperative model by setting the
condition $\beta _{1}=0$ in (\ref{alpha}). For a more thorough examination
of systems with cooperative and competitive structures, we refer to \cite%
{MM1, MM2, MM3}.

The rest of the paper is organized as follows. Section \ref{S2} deals with
the existence of constant-sign solutions for system $(\mathrm{P}),$ while
section \ref{S4} provides a nodal solution.

\section{Two opposite constant-sign solutions}

\label{S2}

In the sequel, the Hilbert spaces $\mathcal{H}^{1}(\Omega )$ and $%
L^{2}(\Omega )$ are equipped with the usual norms $\Vert \cdot \Vert _{1,2}$
and $\Vert \cdot \Vert _{2}$, respectively. We denote by $\mathcal{H}%
_{+}^{1}(\Omega )=\{w\in \mathcal{H}^{1}(\Omega ):w\geq 0$ a.e. in $\Omega
\}.$ We also utilize the H\"{o}lder spaces $\mathcal{C}^{1}(\overline{\Omega 
})$, $\mathcal{C}^{1,\tau }(\overline{\Omega })$ for $\tau \in (0,1),$ $%
\mathcal{C}_{+}^{1}(\overline{\Omega })=\{u\in \mathcal{C}^{1}(\overline{%
\Omega }):u\geq 0$ for all $x\in \overline{\Omega }\}$ and $int\mathcal{C}%
_{+}^{1}(\overline{\Omega })=\{u\in \mathcal{C}^{1}(\overline{\Omega }%
):u(x)>0$ for all $x\in \overline{\Omega }\}$.

Let $\phi _{1}\in int\mathcal{C}_{+}^{1}(\overline{\Omega })$ be the
positive eigenfunction associated with the principal eigenvalue $\lambda
_{1} $ which satisfies%
\begin{equation}
-\Delta \phi _{1}+\phi _{1}=\lambda _{1}\phi _{1}\text{ in }\Omega ,\ \frac{%
\partial \phi _{1}}{\partial \eta }=0\text{ on }\partial \Omega .  \label{4}
\end{equation}%
Set $\mu ,\overline{\mu }>0$ constants such that 
\begin{equation}
\bar{\mu}=\max_{x\in \overline{\Omega }}\phi _{1}(x)\geq \min_{x\in 
\overline{\Omega }}\phi _{1}(x)=\underline{\mu }.  \label{1}
\end{equation}%
Let $w\in intC_{+}^{1}(\overline{\Omega })$ be the solution of Neumann
problem%
\begin{equation}
-\Delta w+w=1\text{ in }\Omega ,\ \frac{\partial w}{\partial \eta }=0\text{
on }\partial \Omega ,  \label{w}
\end{equation}%
which verify 
\begin{equation}
\frac{\phi _{1}}{c_{0}}\leq w\leq c_{0}\phi _{1}\text{ on }\overline{\Omega }%
,  \label{ew}
\end{equation}%
for certain constant $c_{0}>1$ (see \cite{M2}). By comparison principle \cite%
[Lemma 3.2]{ST}, it is readly seen that the solution $y\in int\mathcal{C}%
_{+}^{1}(\overline{\Omega })$ of the homogeneous Neumann problem%
\begin{equation}
-\Delta y+y=1+\rho \text{ in }\Omega ,\ \frac{\partial y}{\partial \eta }=0%
\text{ on }\partial \Omega ,  \label{y}
\end{equation}%
satisfies 
\begin{equation}
\frac{\phi _{1}}{c_{0}}\leq y\leq (1+\rho )c_{0}\phi _{1}\text{ \ on \ }%
\overline{\Omega }.  \label{ey}
\end{equation}

Fix a large constant 
\begin{equation}
C>\max \{1,\frac{1}{\sqrt{\lambda _{1}\overline{\mu }}},\frac{1}{\sqrt{\rho }%
}\}  \label{7}
\end{equation}%
and let $z\in int\mathcal{C}_{+}^{1}(\overline{\Omega })$ be the solution of
Neumann problem \ 
\begin{equation}
-\Delta z+z=C^{-2}\text{ in }\Omega ,\ \frac{\partial z}{\partial \eta }=0%
\text{ on }\partial \Omega ,  \label{z}
\end{equation}%
with%
\begin{equation}
\frac{\phi _{1}}{c_{0}C^{2}}\leq z\leq y\text{ \ on \ }\overline{\Omega }.
\label{ez}
\end{equation}%
Set%
\begin{equation}
(\underline{u},\underline{v})=:(z,z)\text{ \ and \ }(\overline{u},\overline{v%
}):=(Cy,Cy).  \label{subsup}
\end{equation}%
Obviously, $\overline{u}\geq \underline{u}$ and $\overline{v}\geq \underline{%
v}$ in $\Omega .$

Our first result deals with constant-sign solutions, it is stated as follows.

\begin{theorem}
\label{T1}Assume that (\ref{alpha}) holds. Then, problem $(\mathrm{P})$
admits two opposite constant-sign solutions $(u_{+},v_{+})$ and $%
(u_{-},v_{-})$ in $\mathcal{C}^{1}(\overline{\Omega })\times \mathcal{C}^{1}(%
\overline{\Omega }).$ Moreover, if $\beta _{1}=0,$ for a constant $C>1$
large in (\ref{z}), every positive solution $(u_{+},v_{+})$ and negative
solution $(u_{-},v_{-})$ of $(\mathrm{P})$ within $\left[ 0,\overline{u}%
\right] \times \left[ 0,\overline{v}\right] $ and $\left[ -\overline{u},0%
\right] \times \left[ -\overline{v},0\right] $, respectively, satisfy 
\begin{equation}
\begin{array}{c}
\underline{u}(x)<u_{+}(x)\text{ \ and \ }u_{-}(x)<-\underline{u}(x)\text{,}%
\quad \forall x\in \Omega .%
\end{array}
\label{11}
\end{equation}
\end{theorem}

\begin{proof}
Pick $(u,v)\in \mathcal{H}^{1}(\Omega )\times \mathcal{H}^{1}(\Omega )$ such
that $\underline{u}\leq u\leq \overline{u}$ and $\underline{v}\leq v\leq 
\overline{v}$. By (\ref{alpha}), (\ref{ey}), (\ref{ez}) and (\ref{1}), it
follows that%
\begin{eqnarray*}
\frac{\overline{u}^{\alpha _{1}}}{v^{\beta _{1}}} &\leq &\frac{\overline{u}%
^{\alpha _{1}}}{\underline{v}^{\beta _{1}}}\leq \frac{(C(1+\rho )c_{0}\phi
_{1})^{\alpha _{1}}}{(C^{-2}\frac{\phi _{1}}{c_{0}})^{\beta _{1}}}=C^{\alpha
_{1}+2\beta _{1}}c_{0}^{\alpha _{1}+\beta _{1}}(1+\rho )^{\alpha _{1}}\phi
_{1}^{\alpha _{1}-\beta _{1}} \\
&\leq &C^{\alpha _{1}+2\beta _{1}}c_{0}^{\alpha _{1}+\beta _{1}}(1+\rho
)^{\alpha _{1}}\max \{(\bar{\mu})^{\alpha _{1}-\beta _{1}},\underline{\mu }%
^{\alpha _{1}-\beta _{1}}\}
\end{eqnarray*}%
and%
\begin{eqnarray*}
\frac{u^{\alpha _{2}}}{\overline{v}^{\beta _{2}}} &\leq &\frac{\overline{u}%
^{\alpha _{2}}}{\overline{v}^{\beta _{2}}}=(Cy)^{\alpha _{2}-\beta _{2}}\leq
\left\{ 
\begin{array}{ll}
(C(1+\rho )c_{0}\phi _{1})^{\alpha _{2}-\beta _{2}} & \text{if }\alpha
_{2}-\beta _{2}\geq 0 \\ 
(C\frac{\phi _{1}}{c_{0}})^{\alpha _{2}-\beta _{2}} & \text{if }\alpha
_{2}-\beta _{2}\leq 0%
\end{array}%
\right. \\
&\leq &C^{\alpha _{2}-\beta _{2}}\phi _{1}^{\alpha _{2}-\beta _{2}}\max
\{((1+\rho )c_{0})^{\alpha _{2}-\beta _{2}},c_{0}^{-(\alpha _{2}-\beta
_{2})}\} \\
&\leq &C^{\alpha _{2}-\beta _{2}}\max \{\bar{\mu}^{\alpha _{2}-\beta _{2}},%
\underline{\mu }^{\alpha _{2}-\beta _{2}}\}\max \{((1+\rho )c_{0})^{\alpha
_{2}-\beta _{2}},c_{0}^{-(\alpha _{2}-\beta _{2})}\}.
\end{eqnarray*}%
Test with $\varphi _{1},\varphi _{2}\in \mathcal{H}_{+}^{1}(\Omega )$, since 
$\max \{\alpha _{1}+2\beta _{1},\alpha _{2}-\beta _{2}\}<1$ (see (\ref{alpha}%
)), for $C>1$ sufficiently large, we infer that%
\begin{equation}
\int_{\Omega }(\nabla \overline{u}\nabla \varphi _{1}+\overline{u}\varphi
_{1})\ \mathrm{d}x=C\int_{\Omega }(1+\rho )\varphi _{1}\ \mathrm{d}x\geq
\int_{\Omega }f_{1}(v)(\frac{\overline{u}^{\alpha _{1}}}{v^{\beta _{1}}}%
+\rho )\varphi _{1}\,\mathrm{d}x  \label{2}
\end{equation}%
and%
\begin{equation}
\int_{\Omega }(\nabla \overline{v}\nabla \varphi _{2}+\overline{v}\varphi
_{2})\ \mathrm{d}x=C\int_{\Omega }(1+\rho )\varphi _{2}\ \mathrm{d}x\geq
C\int_{\Omega }\varphi _{2}\ \mathrm{d}x\geq \int_{\Omega }f_{2}(u)\frac{%
u^{\alpha _{2}}}{\overline{v}^{\beta _{2}}}\varphi _{2}\,\mathrm{d}x,
\label{2*}
\end{equation}%
showing that $(\overline{u},\overline{v})$ is a positive supersolution pair
for $(\mathrm{P})$.

Next, we show that $(\underline{u},\underline{v})$ in (\ref{subsup}) is a
positive subsolution pair for $(\mathrm{P})$. In view of (\ref{7}), (\ref%
{subsup}) and (\ref{z}), we get%
\begin{equation}
-\Delta \underline{u}+\underline{u}=C^{-2}\leq \rho \leq \frac{\underline{u}%
^{\alpha _{1}}}{v^{\beta _{1}}}+\rho \text{ \ in }\Omega ,  \label{466}
\end{equation}%
{\footnotesize \ }for all $v\in \lbrack \underline{v},\overline{v}]$. By (%
\ref{ey})-(\ref{subsup}), (\ref{alpha}), and after increasing $C$ when
necessary, we obtain%
\begin{equation}
\begin{array}{l}
-\Delta \underline{v}+\underline{v}=C^{-2} \\ 
\leq \left\{ 
\begin{array}{ll}
(\frac{\underline{\mu }}{c_{0}C^{2}})^{\alpha _{2}-\beta _{2}} & \text{if }%
\alpha _{2}-\beta _{2}\geq 0 \\ 
((1+\rho )c_{0}\bar{\mu})^{\alpha _{2}-\beta _{2}} & \text{if }\alpha
_{2}-\beta _{2}\leq 0%
\end{array}%
\right. \\ 
\leq \left\{ 
\begin{array}{ll}
(\frac{\phi _{1}}{c_{0}C^{2}})^{\alpha _{2}-\beta _{2}} & \text{if }\alpha
_{2}-\beta _{2}\geq 0 \\ 
((1+\rho )c_{0}\phi _{1})^{\alpha _{2}-\beta _{2}} & \text{if }\alpha
_{2}-\beta _{2}\leq 0%
\end{array}%
\right. \\ 
\leq z^{\alpha _{2}-\beta _{2}}\leq \frac{\underline{u}^{\alpha _{2}}}{%
\underline{v}^{\beta _{2}}}\leq \frac{u^{\alpha _{2}}}{\underline{v}^{\beta
_{2}}}\text{ \ in }\Omega ,%
\end{array}
\label{46*}
\end{equation}%
for all $u\in \lbrack \underline{u},\overline{u}].${\footnotesize \ }

Test (\ref{466})--(\ref{46*}) with $\varphi _{1},\varphi _{2}\in \mathcal{H}%
_{+}^{1}(\Omega )$ we derive that%
\begin{equation}
\int_{\Omega }(\nabla \underline{u}\nabla \varphi _{1}+\underline{u}\varphi
_{1})\,\mathrm{d}x\leq \int_{\Omega }f_{1}(v)(\frac{\underline{u}^{\alpha
_{1}}}{v^{\beta _{1}}}+\rho )\varphi _{1}\,\mathrm{d}x,  \label{3}
\end{equation}%
\begin{equation}
\int_{\Omega }(\nabla \underline{v}\nabla \varphi _{2}+\underline{v}\varphi
_{2})\,\mathrm{d}x\leq \int_{\Omega }f_{2}(u)\frac{u^{\alpha _{2}}}{%
\underline{v}^{\beta _{2}}}\varphi _{2}\,\mathrm{d}x.  \label{3*}
\end{equation}%
{\footnotesize \ }This shows that $(\underline{u},\underline{v})$ is a
positive subsolutiona pair for $(\mathrm{P}).$ Consequently, on the basis on
(\ref{2}), (\ref{2*}), (\ref{3}) and (\ref{3*}), \cite[Theorem 2.2]{MedjMous}
applies leading to existence of a solution $(u,v)\in \mathcal{C}^{1,\tau }(%
\overline{\Omega })\times \mathcal{C}^{1,\tau }(\overline{\Omega }),$ $\tau
\in (0,1)$, for problem $(\mathrm{P})$ within $[\underline{u},\overline{u}%
]\times \lbrack \underline{v},\overline{v}].$

We proceed to show (\ref{11}). Let $(u_{+},v_{+})\in \left[ 0,\overline{u}%
\right] \times \left[ 0,\overline{v}\right] $ and $(u_{-},v_{-})\in \left[ -%
\overline{u},0\right] \times \left[ -\overline{v},0\right] $ be a positive
and a negative solutions of $(\mathrm{P})$. From (\ref{7}), we have 
\begin{equation*}
u_{+}^{\alpha _{1}}+\rho \geq \rho >C^{-2}\text{ \ in }\Omega ,
\end{equation*}%
and%
\begin{equation*}
-(|u_{-}|^{\alpha _{1}}+\rho )\leq -\rho <-C^{-2}\text{ \ in }\Omega .
\end{equation*}%
Consequently, by the strong maximum principle (see, e.g., \cite{Fan}), we
infer that property (\ref{11}) holds true. This ends the proof.
\end{proof}

\section{A nodal solution}

\label{S4}

This section focuses on nodal solutions for problem $(\mathrm{P})$. The main
result is stated as follows.

\begin{theorem}
\label{T2}Assume (\ref{alpha}) with $\beta _{1}=0$. Then, system $(\mathrm{P}%
)$ possesses nodal solutions $(u^{\ast },v^{\ast })$ in $\mathcal{H}%
^{1}(\Omega )\times \mathcal{H}^{1}(\Omega )$ where components $u^{\ast }$
and $v^{\ast }$ are nontrivial and change sign simultaneously, that is, $%
u^{\ast }v^{\ast }\geq 0.$
\end{theorem}

\begin{remark}
\label{R1}Under assumption (\ref{alpha}) with $\beta _{1}=0$, every solution 
$(u,v)\in \mathcal{H}^{1}(\Omega )\times \mathcal{H}^{1}(\Omega )$ of $(%
\mathrm{P})$ satisfies $u(x),v(x)\neq 0$\ for a.e.\ $x\in \Omega .$
\end{remark}

\subsection{The regularized system}

For all $\varepsilon \in (0,1),$ we state the auxiliary system%
\begin{equation*}
(\mathrm{P}^{\varepsilon })\qquad \left\{ 
\begin{array}{ll}
-\Delta u+u=f_{1}(v)(|u|^{\alpha _{1}}+\rho ) & \text{in }\Omega \\ 
-\Delta v+v=f_{2}(u)\frac{|u|^{\alpha _{2}}}{|v+\gamma _{\varepsilon
}(v)|^{\beta _{2}}} & \text{in }\Omega \\ 
\frac{\partial u}{\partial \eta }=\frac{\partial v}{\partial \eta }=0 & 
\text{on }\partial \Omega ,%
\end{array}%
\right.
\end{equation*}%
where 
\begin{equation*}
\gamma _{\varepsilon }(s)=\varepsilon (\frac{1}{2}+sgn(s)),\;\forall s\in 
\mathbb{\ 
\mathbb{R}
}.
\end{equation*}%
Our goal is to prove that $(\mathrm{P}^{\varepsilon })$ admits a solution $%
(u_{\varepsilon },v_{\varepsilon })$ within $[-\underline{u},\underline{u}%
]\times \lbrack -\underline{v},\underline{v}]$ and then, passing to the
limit as $\varepsilon \rightarrow 0$, we get the existence of the desired
solution $(u^{\ast },v^{\ast })$ for problem $(\mathrm{P})$. The existence
result regarding the regularized system $(\mathrm{P}^{\varepsilon })$ is
stated as follows.

\begin{theorem}
\label{T3} Assume that (\ref{alpha}) hold with $\beta _{1}=0$. Then, the
system $(\mathrm{P}^{\varepsilon })$ possesses solutions $(u_{\varepsilon
},v_{\varepsilon })\in \mathcal{C}^{1,\tau }(\overline{\Omega })\times 
\mathcal{C}^{1,\tau }(\overline{\Omega })$ for some $\tau \in (0,1)$ within $%
\left[ -\underline{u},\underline{u}\right] \times \left[ -\underline{v},%
\underline{v}\right] .$
\end{theorem}

The solution $(u_{\varepsilon },v_{\varepsilon })$ of $(\mathrm{P}%
^{\varepsilon })$ is obtained via topological degree theory. It is located
in the area between the positive and the negative rectangles formed by
positive and negative sub-supersolutions pairs.

For any $R>0$, set 
\begin{equation*}
\mathcal{M}_{R}=\left\{ (u,v)\in \mathcal{B}_{R}(0)\,:-\underline{u}\leq
u\leq \underline{u},\text{ }-\underline{v}\leq v\leq \underline{v}\right\} ,
\end{equation*}%
where $\mathcal{B}_{R}(0)$ denotes the ball in $L^{2}(\Omega )\times
L^{2}(\Omega )$ centered at $0$ of radius $R>0$.

We prove that the degree on a ball $\mathcal{B}_{R_{\varepsilon }}(0)$,
encompassing all potential solutions of $(\mathrm{P}^{\varepsilon }),$ is $0$
while the degree in $\mathcal{B}_{R_{\varepsilon }}(0),$ but excluding the
area located between the aforementioned positive and negative rectangles, is
not zero. By excision property of Leray-Schauder degree, this leads to the
existence of a nontrivial solution $(u_{\varepsilon },v_{\varepsilon })$ for 
$(\mathrm{P}^{\varepsilon }).$

\subsubsection{\textbf{The degree on }$\mathcal{B}_{R_{\protect\varepsilon %
}}(0)$}

Bearing in mind the definition of $\gamma _{\varepsilon }$, we introduce the
truncations%
\begin{equation}
\mathcal{T}_{1}(u(x))=\left\{ 
\begin{array}{ll}
\overline{u}(x) & \text{if \ }u(x)\geq \overline{u}(x) \\ 
u(x) & \text{if \ }-\overline{u}(x)\leq u(x)\leq \overline{u}(x) \\ 
-\overline{u}(x) & \text{if \ }u(x)\leq -\overline{u}(x)%
\end{array}%
\right. ,  \label{19*}
\end{equation}%
\begin{equation}
\mathcal{T}_{2,\varepsilon }(v(x))=\gamma _{\varepsilon }(v(x))+\left\{ 
\begin{array}{ll}
\overline{v}(x) & \text{if \ }v(x)\geq \overline{v}(x) \\ 
v(x) & \text{if \ }-\overline{v}(x)\leq v(x)\leq \overline{v}(x) \\ 
-\overline{v}(x) & \text{if \ }v(x)\leq -\overline{v}(x)%
\end{array}%
\right. ,  \label{19}
\end{equation}%
\ for a.a. $x\in \overline{\Omega },$ for all $\varepsilon \geq 0.$ From the
definition of $\gamma _{\varepsilon }$ and (\ref{subsup}), we derive that%
\begin{equation}
0\leq |\mathcal{T}_{1}(u)|\leq C||y||_{\infty }\text{ \ \ and \ \ }\frac{%
\varepsilon }{2}\leq |\mathcal{T}_{2,\varepsilon }(v)|\leq \frac{%
3\varepsilon }{2}+C||y||_{\infty }.  \label{20}
\end{equation}%
We shall study the homotopy class of problem%
\begin{equation*}
(\mathrm{P}_{t}^{\varepsilon })\qquad \left\{ 
\begin{array}{l}
-\Delta u+u=\mathrm{F}_{1,t}({x,}u,v)\text{ in }\Omega , \\ 
-\Delta v+v=\mathrm{F}_{2,t}^{\varepsilon }({x,}u,v)\text{ in }\Omega , \\ 
\frac{\partial u}{\partial \eta }=\frac{\partial v}{\partial \eta }=0\text{
\ on }\partial \Omega ,%
\end{array}%
\right.
\end{equation*}%
with%
\begin{eqnarray*}
\mathrm{F}_{1,t}(x,u,v) &:&=t\text{ }f_{1}(v)(|\mathcal{T}_{1}(u)|^{\alpha
_{1}}+\rho )+(1-t)(u^{+}+1), \\
\mathrm{F}_{2,t}^{\varepsilon }({x,}u,v) &:&=t\text{ }f_{2}(u)\frac{|%
\mathcal{T}_{1}(u)|^{\alpha _{1}}}{|\mathcal{T}_{2,\varepsilon }(v)|^{\beta
_{1}}}+(1-t)(v^{+}+1)
\end{eqnarray*}%
for $\varepsilon \in (0,1),$ for $t\in \lbrack 0,1],$ where $s^{+}:=\max
\{0,s\}$ and $s^{-}:=\max \{0,-s\},$ for all $s\in 
\mathbb{R}
$. Note that any solution $(u_{\varepsilon },v_{\varepsilon })\in \mathcal{H}%
^{1}(\Omega )\times \mathcal{H}^{1}(\Omega )$ of $(\mathrm{P}%
_{t}^{\varepsilon })$ satisfies $u_{\varepsilon }(x),v_{\varepsilon }(x)\neq
0$ for a.e. $x\in \Omega $. Hence, $\mathrm{F}_{1,t}({x},\cdot ,\cdot )$ and 
$\mathrm{F}_{2,t}^{\varepsilon }{(x},\cdot ,\cdot )$ are continuous for a.e. 
$x\in \Omega ,$ for all $\varepsilon \in (0,1),$ $i=1,2$. Moreover, for $t=0$
in $(\mathrm{P}_{t}^{\varepsilon })$, the decoupled system%
\begin{equation*}
(\mathrm{P}_{0}^{\varepsilon })\qquad \left\{ 
\begin{array}{l}
-\Delta u+u=\mathrm{F}_{1,0}({x,}u,v)=u^{+}+1\text{ in }\Omega , \\ 
-\Delta v+v=\mathrm{F}_{2,0}^{\varepsilon }({x,}u,v)=v^{+}+1\text{ in }%
\Omega , \\ 
\frac{\partial u}{\partial \eta }=\frac{\partial v}{\partial \eta }=0\text{
\ on }\partial \Omega ,%
\end{array}%
\right.
\end{equation*}%
does not admit solutions $(u,v)$ in $\mathcal{H}^{1}(\Omega )\times \mathcal{%
H}^{1}(\Omega )$. This results at once by noting that if the problem admits
a weak solution, then acting with the test function $\varphi \equiv 1$
yields $\int_{\Omega }\ \mathrm{d}x=0$, a contradiction.

On account of (\ref{20}), we derive the estimate%
\begin{equation*}
|\mathrm{F}_{1,t}(x,u,v)|\leq C(1+|u|)\text{ and }|\mathrm{F}%
_{2,t}^{\varepsilon }({x,}u,v)|\leq C_{\varepsilon }^{\prime }(1+|v|),\text{
for a.e. }x\in \Omega ,
\end{equation*}%
for certain constants $C,C_{\varepsilon }^{\prime }>0$. Then, according to 
\cite[Corollary 8.13]{MMP2}, we conclude that each solution $(u_{\varepsilon
},v_{\varepsilon })$ of $(\mathrm{P}_{t}^{\varepsilon })$ belongs to $%
\mathcal{C}^{1}(\overline{\Omega })\times \mathcal{C}^{1}(\overline{\Omega }%
) $ and there exists a constant $R_{\varepsilon }>0$ such that 
\begin{equation}
\left\Vert u_{\varepsilon }\right\Vert _{\mathcal{C}^{1}(\overline{\Omega }%
)},\left\Vert v_{\varepsilon }\right\Vert _{\mathcal{C}^{1}(\overline{\Omega 
})}<R_{\varepsilon },  \label{21}
\end{equation}%
for all $t\in (0,1]$ and $\varepsilon \in (0,1)$.

For every $\varepsilon \in (0,1)$, let us define the homotopy $\mathcal{H}%
_{\varepsilon }:[0,1]\times \mathcal{B}_{R_{\varepsilon }}(0)\rightarrow
L^{2}(\Omega )\times L^{2}(\Omega )$ by%
\begin{equation*}
\mathcal{H}_{\varepsilon }(t,u,v)=I(u,v)-\left( 
\begin{array}{cc}
(-\Delta +I)^{-1} & 0 \\ 
0 & (-\Delta +I)^{-1}%
\end{array}%
\right) \left( 
\begin{array}{l}
\mathrm{F}_{1,t}({x,}u,v) \\ 
\multicolumn{1}{c}{\mathrm{F}_{2,t}^{\varepsilon }({x,}u,v)}%
\end{array}%
\right) .
\end{equation*}%
that is admissible for the Leray-Schauder topological degree by (\ref{21}),
the continuity of $\mathrm{F}{_{1,t}(x},\cdot ,\cdot )$ and $\mathrm{F}{%
_{2,t}^{\varepsilon }(x},\cdot ,\cdot )$ for a.e. $x\in \Omega $ and because
the operator $(-\Delta +I)^{-1},$ with values in $L^{2}(\Omega ),$ is
compact. Note that $(u_{\varepsilon },v_{\varepsilon })\in \mathcal{B}%
_{R_{\varepsilon }}(0)$ is a solution for $(\mathrm{P}^{\varepsilon })$ if,
and only if, 
\begin{equation*}
(u_{\varepsilon },v_{\varepsilon })\in \mathcal{B}_{R_{\varepsilon }}(0)\,\ 
\text{and\thinspace }\mathcal{\ H}_{\varepsilon }(1,u_{\varepsilon
},v_{\varepsilon })=0.
\end{equation*}%
The a priori estimate (\ref{21}) establishes expressly that solutions of $(%
\mathrm{P}_{t}^{\varepsilon })$ must lie in $\mathcal{B}_{R_{\varepsilon
}}(0)$, while the nonexistence of solutions to problem $(\mathrm{P}%
_{0}^{\varepsilon })$ yields $\deg \left( \mathcal{H}_{\varepsilon }(0,\cdot
,\cdot ),\mathcal{B}_{R_{\varepsilon }}(0),0\right) =0,$ for all $%
\varepsilon \in (0,1).$ Consequently, the homotopy invariance property of
the degree implies that 
\begin{equation}
\begin{array}{c}
\deg \left( \mathcal{H}_{\varepsilon }(1,\cdot ,\cdot ),\mathcal{B}%
_{R_{\varepsilon }}(0),0\right) =0,\text{ for all }\varepsilon \in (0,1).%
\end{array}
\label{22}
\end{equation}

\subsubsection{\textbf{The degree on }$\mathcal{B}_{R_{\protect\varepsilon %
}}(0)\backslash \overline{\mathcal{M}_{R_{\protect\varepsilon }}}$\textbf{.}}

We show that the degree of an operator corresponding to problem $(\mathrm{P}%
^{\varepsilon })$ is not zero outside the set $\mathcal{M}_{R_{\varepsilon
}} $. To this end, let us define the problem%
\begin{equation*}
(\widehat{\mathrm{P}}_{t}^{\varepsilon })\qquad \left\{ 
\begin{array}{l}
-\Delta u+u=\widehat{\mathrm{F}}_{1,t}({x,}u,v)\text{ in }\Omega , \\ 
-\Delta v+v=\widehat{\mathrm{F}}_{2,t}^{\varepsilon }({x,}u,v)\text{ in }%
\Omega , \\ 
\frac{\partial u}{\partial \eta }=\frac{\partial v}{\partial \eta }=0\text{
\ on }\partial \Omega ,%
\end{array}%
\right.
\end{equation*}%
for $t\in \lbrack 0,1]$ and $\varepsilon \in (0,1)$, where%
\begin{eqnarray*}
\widehat{\mathrm{F}}_{1,t}({x,}u,v) &:&=tf_{1}(v)(|\mathcal{T}%
_{1}(u)|^{\alpha _{1}}+\rho )+\frac{2}{3}(1-t)\lambda _{1}\hat{\chi}_{\phi
_{1}}(u) \\
\widehat{\mathrm{F}}{_{2,t}^{\varepsilon }}({x,}u,v) &:&=t\text{ }f_{2}(u)%
\frac{|\mathcal{T}_{1}(u)|^{\alpha _{2}}}{|\mathcal{T}_{2,\varepsilon
}(v)|^{\beta _{2}}}+\frac{2}{3}(1-t)\lambda _{1}\hat{\chi}_{\phi _{1}}(v),
\end{eqnarray*}%
where the truncation $\hat{\chi}_{\phi _{1}}$ is defined by%
\begin{equation}
\hat{\chi}_{\phi _{1}}(s)=\left\{ 
\begin{array}{ll}
\frac{3}{2}s & \text{if }s\geq \phi _{1} \\ 
(\frac{1}{2}+sgn(s))\text{ }\phi _{1} & \text{if }-\phi _{1}\leq s\leq \phi
_{1} \\ 
\frac{1}{2}s & \text{if }s\leq -\phi _{1}.%
\end{array}%
\right.  \label{23}
\end{equation}%
Note that every solution $(u,v)\in \mathcal{H}^{1}(\Omega )\times \mathcal{H}%
^{1}(\Omega )$ of $(\widehat{\mathrm{P}}_{t}^{\varepsilon })$ satisfies $%
u(x),v(x)\neq 0$ for a.e. $x\in \Omega $. This leads to conclude that $%
\widehat{\mathrm{F}}_{1,t}{(x},\cdot ,\cdot )$ and $\widehat{\mathrm{F}}{%
_{2,t}^{\varepsilon }(x},\cdot ,\cdot )$ are continuous for a.e. $x\in
\Omega ,$ for all $\varepsilon \in (0,1)$.

We show that solutions of problem $(\widehat{\mathrm{P}}_{t}^{\varepsilon })$
cannot occur outside the ball $\mathcal{B}_{R_{\varepsilon }}(0)$.

\begin{proposition}
\label{P3} Assume that (\ref{alpha}) is fulfilled with $\beta _{1}=0$.\
Then, any solution $(u,v)$ of $(\widehat{\mathrm{P}}_{t}^{\varepsilon })$
belongs to $\mathcal{C}^{1}(\overline{\Omega })\times \mathcal{C}^{1}(%
\overline{\Omega })$ and satisfy 
\begin{equation}
\left\Vert u\right\Vert _{\mathcal{C}^{1}(\overline{\Omega })},\left\Vert
v\right\Vert _{\mathcal{C}^{1}(\overline{\Omega })}<R_{\varepsilon },
\label{24}
\end{equation}%
for $t\in \lbrack 0,1]$ and $\varepsilon \in (0,1)$. In addition, all
positive and negative solutions $(u_{+},v_{+})$ and $(u_{-},v_{-})$ of $(%
\widehat{\mathrm{P}}_{t}^{\varepsilon })$ satisfy 
\begin{equation}
\begin{array}{l}
u_{+}(x)>\underline{u}(x),\text{ \ \ }v_{+}(x)>\underline{v}(x) \\ 
-\underline{u}(x)>u_{-}(x),\text{ }-\underline{v}(x)>v_{-}(x)%
\end{array}%
\quad \text{for all }x\in \Omega .  \label{25}
\end{equation}
\end{proposition}

\begin{proof}
Let $(u,v)\in \mathcal{H}^{1}(\Omega )\times \mathcal{H}^{1}(\Omega )$ be a
solution of $(\widehat{\mathrm{P}}_{t}^{\varepsilon })$. From (\ref{23}) and
(\ref{7}), one has 
\begin{equation*}
\frac{2}{3}\hat{\chi}_{\phi _{1}}(u)\leq \max \{u,\phi _{1}\}\text{ \ and \ }%
\frac{2}{3}\hat{\chi}_{\phi _{1}}(v)\leq \max \{v,\phi _{1}\}.
\end{equation*}%
Thus, by (\ref{20}) we get%
\begin{equation*}
|\widehat{\mathrm{F}}_{1,t}({x,}u,v)|\leq c+\lambda _{1}\max \{u,\phi _{1}\}
\end{equation*}%
and 
\begin{equation*}
|\widehat{\mathrm{F}}{_{2,t}^{\varepsilon }}({x,}u,v)|\leq c_{\varepsilon
}+\lambda _{1}\max \{v,\phi _{1}\},
\end{equation*}%
for all $\varepsilon \in (0,1),$ where $c,c_{\varepsilon }>0$ are certain
constants. Then,\ the regularity theory up to the boundary (see \cite[%
Corollary 8.13]{MMP2}) together with the compact embedding $\mathcal{C}%
^{1,\tau }(\overline{\Omega })\subset \mathcal{C}^{1}(\overline{\Omega })$
entails the bound in (\ref{24}), for all $\varepsilon \in (0,1)$.

We proceed to show the inequalities in (\ref{25}). Let $(u,v)$ be a positive
solution of $(\widehat{\mathrm{P}}_{t}^{\varepsilon }).$ By (\ref{1}) and
after increasing $C>1$ when necessary, it follows that%
\begin{equation*}
\begin{array}{l}
\widehat{\mathrm{F}}_{1,t}({x,}u,v)>t\rho +(1-t)\lambda _{1}\phi _{1} \\ 
\geq t\rho +(1-t)\lambda _{1}\underline{\mu }>C^{-2}\text{ in }\Omega .%
\end{array}%
\end{equation*}%
Thus, (\ref{z}) and (\ref{subsup}) together with the strong maximum
principle (see, e.g., \cite{Fan}) impply that%
\begin{equation}
u(x)>\underline{u}(x)\text{ for all }x\in \Omega .  \label{8}
\end{equation}%
By (\ref{alpha}), (\ref{20}), (\ref{1}), (\ref{23}) and (\ref{ez}),
increasing $C>1$ when necessary, we get 
\begin{equation*}
\begin{array}{l}
\widehat{\mathrm{F}}{_{2,t}^{\varepsilon }}({x,}u,v)=t\frac{|\mathcal{T}%
_{1}(u)|^{\alpha _{2}}}{|\mathcal{T}_{2,\varepsilon }(v)|^{\beta _{2}}}+%
\frac{2}{3}(1-t)\lambda _{1}\hat{\chi}_{\phi _{1}}(v) \\ 
\geq t\frac{\underline{u}^{\alpha _{2}}}{(\frac{3\varepsilon }{2}%
+C||y||_{\infty })^{\beta _{2}}}+(1-t)\lambda _{1}\phi _{1}\geq t\frac{(%
\frac{\phi _{1}}{c_{0}C^{2}})^{\alpha _{2}}}{(\frac{3}{2}+C||y||_{\infty
})^{\beta _{2}}}+(1-t)\lambda _{1}\underline{\mu } \\ 
\geq tC^{-(2\alpha _{2}+\beta _{2})}\frac{(\frac{\underline{\mu }}{c_{0}}%
)^{\alpha _{2}}}{(\frac{3}{2}+||y||_{\infty })^{\beta _{2}}}+(1-t)\lambda
_{1}\underline{\mu }>C^{-2}\text{ in }\Omega .%
\end{array}%
\end{equation*}%
Again, by (\ref{z}), (\ref{subsup}) and, the strong maximum principle, we
derive that%
\begin{equation*}
v(x)>\underline{v}(x),\text{ for all }x\in \Omega .
\end{equation*}%
A quite similar argument shows that 
\begin{equation*}
-\underline{u}(x)>u_{-}(x),\text{ \ }-\underline{v}(x)>v_{-}(x)\text{ \ for
all }x\in \Omega .
\end{equation*}
\end{proof}

Let us define the homotopy $\mathcal{N}_{\varepsilon }$ on $[0,1]\times 
\mathcal{B}_{R_{\varepsilon }}(0)\backslash \overline{\mathcal{M}%
_{R_{\varepsilon }}}\rightarrow L^{2}(\Omega )\times L^{2}(\Omega )$ by%
\begin{equation}
\mathcal{N}_{\varepsilon }(t,u,v)=I(u,v)-\left( 
\begin{array}{cc}
(-\Delta +I)^{-1} & 0 \\ 
0 & (-\Delta +I)^{-1}%
\end{array}%
\right) \left( 
\begin{array}{l}
\widehat{\mathrm{F}}{_{1,t}}({x,}u,v) \\ 
\multicolumn{1}{c}{\widehat{\mathrm{F}}{_{2,t}^{\varepsilon }}({x,}u,v)}%
\end{array}%
\right) .  \label{33}
\end{equation}%
for $t\in \lbrack 0,1]$ and $\varepsilon \in (0,1)$. Clearly, $\mathcal{N}%
_{\varepsilon }$ is well defined, compact and continuous a.e. in $\Omega $.
Moreover, $(u,v)\in \mathcal{B}_{R_{\varepsilon }}(0)\backslash \overline{%
\mathcal{M}_{R_{\varepsilon }}}$ is a solution of system $(\mathrm{\ P}%
^{\varepsilon })$ if, and only if, 
\begin{equation*}
\begin{array}{c}
(u,v)\in \mathcal{B}_{R_{\varepsilon }}(0)\backslash \overline{\mathcal{M}%
_{R_{\varepsilon }}}\,\,\,\mbox{and}\,\,\,\mathcal{N}_{\varepsilon
}(1,u,v)=0.%
\end{array}%
\end{equation*}%
In view of (\ref{4}), (\ref{1}), (\ref{7}) and (\ref{z}), $\phi _{1}\in 
\mathcal{B}_{R_{\varepsilon }}(0)\backslash \overline{\mathcal{M}%
_{R_{\varepsilon }}}$ which, by (\ref{23}) and (\ref{4}), is actually the
unique solution of the problem 
\begin{equation*}
-\Delta w+w=\frac{2}{3}\lambda _{1}\hat{\chi}_{\phi _{1}}(w)\text{ in }%
\Omega ,\text{ }\frac{\partial w}{\partial \eta }=0\text{ on }\partial
\Omega .
\end{equation*}%
Then, the homotopy invariance property of the degree gives 
\begin{equation}
\deg (\mathcal{N}_{\varepsilon }(1,\cdot ,\cdot ),\mathcal{B}%
_{R_{\varepsilon }}(0)\backslash \overline{\mathcal{M}_{R_{\varepsilon }}}%
,0)=\deg (\mathcal{N}_{\varepsilon }(0,\cdot ,\cdot ),\mathcal{B}%
_{R_{\varepsilon }}(0)\backslash \overline{\mathcal{M}_{R_{\varepsilon }}}%
,0)\neq 0.  \label{34}
\end{equation}%
Since 
\begin{equation*}
\mathcal{H}_{\varepsilon }(1,\cdot ,\cdot )=\mathcal{N}_{\varepsilon
}(1,\cdot ,\cdot )\,in\,\mathcal{B}_{R_{\varepsilon }}(0)\backslash 
\overline{\mathcal{M}_{R_{\varepsilon }}},\text{ for all }\varepsilon \in
(0,1),
\end{equation*}
we deduce that 
\begin{equation}
\begin{array}{c}
\deg (\mathcal{H}_{\varepsilon }(1,\cdot ,\cdot ),\mathcal{B}%
_{R_{\varepsilon }}(0)\backslash \overline{\mathcal{M}_{R_{\varepsilon }}}%
,0)\neq 0.%
\end{array}
\label{35}
\end{equation}

\subsubsection{\textbf{Proof of Theorem \protect\ref{T3}.}}

We assume that $\mathcal{H}_{\varepsilon }(1,u,v)\not=0,$ for all $(u,v)\in
\partial \mathcal{M}_{R_{\varepsilon }},$ for all $\varepsilon \in (0,1).$
Otherwise, $(u,v)\in \partial \mathcal{M}_{R_{\varepsilon }}$ would be a
solution of $(\mathrm{P}^{\varepsilon })$ within $\left[ -\underline{u},%
\underline{u}\right] \times \left[ -\underline{v},\underline{v}\right] $ and
thus, Theorem \ref{T3} is proved.

By virtue of the domain additivity property of Leray-Schauder degree it
follows that 
\begin{eqnarray*}
&&\deg (\mathcal{H}_{\varepsilon }(1,\cdot ,\cdot ),\mathcal{B}%
_{R_{\varepsilon }}(0),0) \\
&=&\deg (\mathcal{H}_{\varepsilon }(1,\cdot ,\cdot ),\mathcal{B}%
_{R_{\varepsilon }}(0)\backslash \overline{\mathcal{M}_{R_{\varepsilon }}}%
,0)+\deg (\mathcal{H}_{\varepsilon }(1,\cdot ,\cdot ),\mathcal{M}%
_{R_{\varepsilon }},0).
\end{eqnarray*}%
Hence, by (\ref{22}) and (\ref{35}), we deduce that $\deg (\mathcal{H}%
_{\varepsilon }(1,\cdot ,\cdot ),\mathcal{M}_{R_{\varepsilon }},0)\neq 0,$
showing that problem $(\mathrm{P}^{\varepsilon })$ has a solution $%
(u_{\varepsilon },v_{\varepsilon })\in \mathcal{M}_{R_{\varepsilon }},$ for
all $\varepsilon \in (0,1)$. The nonlinear regularity theory \cite{L}
guarantees that $(u_{\varepsilon },v_{\varepsilon })\in \mathcal{C}^{1,\tau
}(\overline{\Omega })\times \mathcal{C}^{1,\tau }(\overline{\Omega })$ for
certain $\tau \in (0,1)$.

\subsection{\textbf{Proof of Theorem \protect\ref{T2}.}}

Set $\varepsilon =\frac{1}{n}$ in $(\mathrm{P}^{\varepsilon })$ with any
positive integer $n\geq 1.$ According to Theorem \ref{T3}, there exists $%
(u_{n},v_{n}):=(u_{\frac{1}{n}},v_{\frac{1}{n}})\in \mathcal{C}^{1,\tau }(%
\overline{\Omega })\times \mathcal{C}^{1,\tau }(\overline{\Omega })$
solution of $(\mathrm{P}^{n})$ ($(\mathrm{P}^{\varepsilon })$ with $%
\varepsilon =\frac{1}{n}$) such that 
\begin{equation*}
(u_{n},v_{n})\in \left[ -\underline{u},\underline{u}\right] \times \left[ -%
\underline{v},\underline{v}\right]
\end{equation*}%
and%
\begin{equation}
\left\{ 
\begin{array}{l}
\int_{\Omega }(\nabla u_{n}\text{\thinspace }\nabla \varphi
_{1}+u_{n}\varphi _{1})\text{ }\mathrm{d}x=\int_{\Omega
}f_{1}(v_{n})(|u_{n}|^{\alpha _{1}}+\rho )\varphi _{1}\text{ }\mathrm{d}x,
\\ 
\int_{\Omega }(\nabla v_{n}\text{\thinspace }\nabla \varphi
_{2}+v_{n}\varphi _{2})\text{ }\mathrm{d}x=\int_{\Omega }f_{2}(u_{n})\frac{%
|u_{n}|^{\alpha _{2}}}{|v_{n}+\gamma _{n}(v_{n})|^{\beta _{2}}}\varphi _{2}%
\text{ }\mathrm{d}x,%
\end{array}%
\right.  \label{55}
\end{equation}%
for all $\varphi _{i}\in \mathcal{H}^{1}(\Omega ),$ $i=1,2$, where $\gamma
_{n}(\cdot ):=\gamma _{\frac{1}{n}}(\cdot )$. Passing to relabeled
subsequences, the compact embedding $\mathcal{C}^{1,\tau }(\overline{\Omega }%
)\hookrightarrow \mathcal{C}^{1}(\overline{\Omega })$ entails the strong
convergence $(u_{n},v_{n})\rightarrow (u^{\ast },v^{\ast })$ in$\;\mathcal{C}%
^{1}(\overline{\Omega })\times \mathcal{C}^{1}(\overline{\Omega })$ and
therefore, 
\begin{equation}
(u_{n},v_{n})\rightarrow (u^{\ast },v^{\ast })\text{\ in}\;\mathcal{H}%
^{1}(\Omega )\times \mathcal{H}^{1}(\Omega ).  \label{38}
\end{equation}%
Young inequality implies%
\begin{equation}
\begin{array}{c}
\int_{\Omega }(\nabla u^{\ast }\text{\thinspace }\nabla \varphi _{1}+u^{\ast
}\varphi _{1})\text{ }\mathrm{d}x\leq \frac{1}{2}\left\Vert \nabla u^{\ast
}\right\Vert _{2}^{2}+\frac{1}{2}\left\Vert \nabla \varphi _{1}\right\Vert
_{2}^{2}+\frac{1}{2}\left\Vert u^{\ast }\right\Vert _{2}^{2}+\frac{1}{2}%
\left\Vert \varphi _{1}\right\Vert _{2}^{2} \\ 
\leq \left\Vert u^{\ast }\right\Vert _{1,2}^{2}+\left\Vert \nabla \varphi
_{1}\right\Vert _{1,2}^{2},%
\end{array}
\label{41}
\end{equation}%
\begin{equation}
\begin{array}{c}
\int_{\Omega }(\nabla v^{\ast }\text{\thinspace }\nabla \varphi _{2}+v^{\ast
}\varphi _{2})\text{ }\mathrm{d}x\leq \frac{1}{2}\left\Vert \nabla v^{\ast
}\right\Vert _{2}^{2}+\frac{1}{2}\left\Vert \nabla \varphi _{2}\right\Vert
_{2}^{2}+\frac{1}{2}\left\Vert v^{\ast }\right\Vert _{2}^{2}+\frac{1}{2}%
\left\Vert \varphi _{2}\right\Vert _{2}^{2} \\ 
\leq \left\Vert v^{\ast }\right\Vert _{1,2}^{2}+\left\Vert \nabla \varphi
_{2}\right\Vert _{1,2}^{2},%
\end{array}%
\end{equation}%
for all $\varphi _{i}\in \mathcal{H}^{1}(\Omega ),$ $i=1,2.$ Moreover,
Lebesgue's dominated convergence theorem entails%
\begin{equation}
\lim_{n\rightarrow +\infty }\int_{\Omega }(\nabla u_{n}\text{\thinspace }%
\nabla \varphi _{1}+u_{n}\varphi _{1})\text{ }\mathrm{d}x=\int_{\Omega
}(\nabla u^{\ast }\text{\thinspace }\nabla \varphi _{1}+u^{\ast }\varphi
_{1})\text{ }\mathrm{d}x,  \label{40}
\end{equation}%
\begin{equation}
\lim_{n\rightarrow +\infty }\int_{\Omega }(\nabla v_{n}\text{\thinspace }%
\nabla \varphi _{2}+v_{n}\varphi _{2})\text{ }\mathrm{d}x=\int_{\Omega
}(\nabla v^{\ast }\text{\thinspace }\nabla \varphi _{2}+v^{\ast }\varphi
_{2})\text{ }\mathrm{d}x  \label{40*}
\end{equation}%
and%
\begin{equation}
\lim_{n\rightarrow +\infty }\int_{\Omega }f_{1}(v_{n})|u_{n}|^{\alpha
_{1}}+\rho )\varphi _{1}\text{ }\mathrm{d}x=\int_{\Omega }f_{1}(v^{\ast
})|u^{\ast }|^{\alpha _{1}}+\rho )\varphi _{1}\text{ }\mathrm{d}x,
\label{57}
\end{equation}%
for all $\varphi _{i}\in \mathcal{H}^{1}(\Omega )$, $i=1,2.$ Let us we show
that 
\begin{equation}
\lim_{n\rightarrow +\infty }\int_{\Omega }f_{2}(u_{n})\frac{|u_{n}|^{\alpha
_{2}}}{|v_{n}+\gamma _{n}(v_{n})|^{\beta _{2}}}\varphi _{2}\text{ }\mathrm{d}%
x=\int_{\Omega }f_{2}(u^{\ast })\frac{|u^{\ast }|^{\alpha _{2}}}{|v^{\ast
}|^{\beta _{2}}}\varphi _{2}\text{ }\mathrm{d}x,  \label{42}
\end{equation}%
for all $\varphi _{2}\in \mathcal{H}^{1}(\Omega ).$ Assume $\varphi _{2}\geq
0$ in $\Omega $ and write%
\begin{equation}
\int_{\Omega }f_{2}(u^{\ast })\frac{|u^{\ast }|^{\alpha _{2}}}{|v^{\ast
}|^{\beta _{2}}}\varphi _{2}\text{ }\mathrm{d}x=\int_{\Omega }\frac{%
|(u^{\ast })^{+}|^{\alpha _{2}}}{|v^{\ast }|^{\beta _{2}}}\varphi _{2}\text{ 
}\mathrm{d}x-\int_{\Omega }\frac{|(u^{\ast })^{-}|^{\alpha _{2}}}{|v^{\ast
}|^{\beta _{2}}}\varphi _{2}\text{ }\mathrm{d}x.  \label{37}
\end{equation}%
Given that $\frac{|s|^{\alpha _{2}}}{|t|^{\beta _{2}}}$ is a continuous
function for $(s,t)\in (%
\mathbb{R}
\backslash \{0\})^{2}$, Fatou's Lemma along with (\ref{38}) imply%
\begin{eqnarray*}
\int_{\Omega }\frac{|(u^{\ast })^{+}|^{\alpha _{2}}}{|v^{\ast }|^{\beta _{2}}%
}\varphi _{2}\text{ }\mathrm{d}x &\leq &\int_{\Omega }\lim_{n\rightarrow
+\infty }\inf (\frac{|(u_{n})^{+}|^{\alpha _{2}}}{|v_{n}+\gamma
_{n}(v_{n})|^{\beta _{2}}}\varphi _{2})\text{ }\mathrm{d}x \\
&\leq &\lim_{n\rightarrow +\infty }\inf \int_{\Omega }\frac{%
|(u_{n})^{+}|^{\alpha _{2}}}{|v_{n}+\gamma _{n}(v_{n})|^{\beta _{2}}}\varphi
_{2}\text{ }\mathrm{d}x,
\end{eqnarray*}%
as well as%
\begin{eqnarray*}
\int_{\Omega }\frac{|(u^{\ast })^{-}|^{\alpha _{2}}}{|v^{\ast }|^{\beta _{2}}%
}\varphi _{2}\text{ }\mathrm{d}x &\geq &\int_{\Omega }\lim_{n\rightarrow
+\infty }\sup (\frac{|(u_{n})^{-}|^{\alpha _{2}}}{|v_{n}+\gamma
_{n}(v_{n})|^{\beta _{2}}}\varphi _{2})\text{ }\mathrm{d}x \\
&\geq &\lim_{n\rightarrow +\infty }\sup \int_{\Omega }\frac{%
|(u_{n})^{-}|^{\alpha _{2}}}{|v_{n}+\gamma _{n}(v_{n})|^{\beta _{2}}}\varphi
_{2}\text{ }\mathrm{d}x.
\end{eqnarray*}%
Then, using (\ref{37}), (\ref{41}), (\ref{55}) and (\ref{40*}), it follows
that%
\begin{equation*}
\begin{array}{l}
\int_{\Omega }f_{2}(u^{\ast })\frac{|u^{\ast }|^{\alpha _{2}}}{|v^{\ast
}|^{\beta _{2}}}\varphi _{2}\text{ }\mathrm{d}x \\ 
\leq \lim_{n\rightarrow +\infty }\inf \int_{\Omega }\frac{%
|(u_{n})^{+}|^{\alpha _{2}}}{|v_{n}+\gamma _{n}(v_{n})|^{\beta _{2}}}\varphi
_{2}\text{ }\mathrm{d}x-\lim_{n\rightarrow +\infty }\sup \int_{\Omega }\frac{%
|(u_{n})^{-}|^{\alpha _{2}}}{|v_{n}+\gamma _{n}(v_{n})|^{\beta _{2}}}\varphi
_{2}\text{ }\mathrm{d}x \\ 
\leq \lim_{n\rightarrow +\infty }\int_{\Omega }\frac{|(u_{n})^{+}|^{\alpha
_{2}}}{|v_{n}+\gamma _{n}(v_{n})|^{\beta _{2}}}\varphi _{2}\text{ }\mathrm{d}%
x-\lim_{n\rightarrow +\infty }\int_{\Omega }\frac{|(u_{n})^{-}|^{\alpha _{2}}%
}{|v_{n}+\gamma _{n}(v_{n})|^{\beta _{2}}}\varphi _{2}\text{ }\mathrm{d}x \\ 
=\lim_{n\rightarrow +\infty }\int_{\Omega }f_{2}(u_{n})\frac{|u_{n}|^{\alpha
_{2}}}{|v_{n}+\gamma _{n}(v_{n})|^{\beta _{2}}}\varphi _{2}\text{ }\mathrm{d}%
x \\ 
\leq \left\Vert v^{\ast }\right\Vert _{1,2}^{2}+\left\Vert \nabla \varphi
_{2}\right\Vert _{1,2}^{2},%
\end{array}%
\end{equation*}%
showing that%
\begin{equation}
f_{2}(u^{\ast })\frac{|u^{\ast }|^{\alpha _{2}}}{|v^{\ast }|^{\beta _{2}}}%
\varphi _{2}\in L^{1}(\Omega ),\text{ for all }\varphi _{2}\in \mathcal{H}%
^{1}(\Omega )\text{ with }\varphi _{2}\geq 0\text{ in }\Omega .  \label{43}
\end{equation}%
For a fixed $\mu >0$, we write 
\begin{equation}
\begin{array}{l}
\int_{\Omega }f_{2}(u_{n})\frac{|u_{n}|^{\alpha _{2}}}{|v_{n}+\gamma
_{n}(v_{n})|^{\beta _{2}}}\varphi _{2}\text{ }\mathrm{d}x \\ 
=\int_{\Omega \cap \{|v_{n}|\leq \mu \}}f_{2}(u_{n})\frac{|u_{n}|^{\alpha
_{2}}}{|v_{n}+\gamma _{n}(v_{n})|^{\beta _{2}}}\varphi _{2}\text{ }\mathrm{d}%
x+\int_{\Omega \cap \{|v_{n}|>\mu \}}f_{2}(u_{n})\frac{|u_{n}|^{\alpha _{2}}%
}{|v_{n}+\gamma _{n}(v_{n})|^{\beta _{2}}}\varphi _{2}\text{ }\mathrm{d}x.%
\end{array}
\label{44}
\end{equation}%
Define the truncation $\chi _{\mu }:%
\mathbb{R}
\rightarrow \lbrack 0,+\infty \lbrack $ by 
\begin{equation*}
\chi _{\mu }(s)=\left\{ 
\begin{array}{ll}
0 & \text{if }|s|\geq 2\mu , \\ 
2-sgn(s)\frac{s}{\mu } & \text{if }\mu \leq |s|\leq 2\mu , \\ 
1 & \text{if }|s|\leq \mu .%
\end{array}%
\right.
\end{equation*}%
Test in (\ref{55}) with $\chi _{\mu }(v_{n}^{+})\varphi _{2}\in \mathcal{H}%
^{1}(\Omega ),$ which is possible due to the continuity of function $\chi
_{\mu },$ reads as%
\begin{equation}
\int_{\Omega }(\nabla v_{n}\text{\thinspace }\nabla (\chi _{\mu
}(v_{n}^{+})\varphi _{2})+v_{n}\chi _{\mu }(v_{n}^{+})\varphi _{2})\text{ }%
\mathrm{d}x=\int_{\Omega }f_{2}(u_{n})\frac{|u_{n}|^{\alpha _{2}}}{%
|v_{n}+\gamma _{n}(v_{n})|^{\beta _{2}}}\chi _{\mu }(v_{n}^{+})\varphi _{2}%
\text{ }\mathrm{d}x.  \label{45}
\end{equation}%
By definition of $\chi _{\mu }$ we get 
\begin{equation}
\int_{\Omega }|\nabla v_{n}|^{2}\chi _{\mu }^{\prime }(v_{n}^{+})\varphi _{2}%
\text{ }\mathrm{d}x=-\frac{1}{\mu }\int_{\Omega }|\nabla v_{n}|^{2}\varphi
_{2}\text{ }\mathrm{d}x.  \label{46}
\end{equation}%
Thence%
\begin{equation}
\int_{\Omega }(\nabla v_{n}\text{\thinspace }\nabla (\chi _{\mu
}(v_{n}^{+})\varphi _{2})+v_{n}\chi _{\mu }(v_{n}^{+})\varphi _{2})\text{ }%
\mathrm{d}x\leq \int_{\Omega }(\nabla v_{n}\text{\thinspace }\nabla \varphi
_{2}\text{ }\chi _{\mu }(v_{n}^{+})+v_{n}\chi _{\mu }(v_{n}^{+})\varphi _{2})%
\text{ }\mathrm{d}x,  \label{47}
\end{equation}%
which, by (\ref{38}) together with Lebesgue's Theorem, gives%
\begin{equation}
\begin{array}{l}
\lim_{n\rightarrow +\infty }\int_{\Omega }(\nabla v_{n}\text{\thinspace }%
\nabla \varphi _{2}\text{ }\chi _{\mu }(v_{n}^{+})+v_{n}\chi _{\mu
}(v_{n}^{+})\varphi _{2})\text{ }\mathrm{d}x \\ 
\\ 
\leq \int_{\Omega }(\nabla v^{\ast }\text{\thinspace }\nabla \varphi _{2}%
\text{ }\chi _{\mu }((v^{\ast })^{+})+v^{\ast }\chi _{\mu }((v^{\ast
})^{+})\varphi _{2})\text{ }\mathrm{d}x.%
\end{array}
\label{48*}
\end{equation}%
Repeating the previous argument by testing in (\ref{55}) with $\chi _{\mu
}(-v_{n}^{-})\varphi _{2}\in \mathcal{H}^{1}(\Omega )$, we get%
\begin{equation}
\begin{array}{l}
\lim_{n\rightarrow +\infty }\int_{\Omega }\nabla v_{n}^{-}\text{\thinspace }%
\nabla \varphi _{2}\text{ }\chi _{\mu }(-v_{n}^{-})+v_{n}\chi _{\mu
}(-v_{n}^{-})\varphi _{2})\text{ }\mathrm{d}x \\ 
\\ 
\leq \int_{\Omega }(\nabla v^{\ast }\text{\thinspace }\nabla \varphi _{2}%
\text{ }\chi _{\mu }(-(v^{\ast })^{-})+v^{\ast }\chi _{\mu }(-(v^{\ast
})^{-})\varphi _{2})\text{ }\mathrm{d}x.%
\end{array}
\label{49}
\end{equation}%
Note from the definition of $\chi _{\mu }(\cdot )$ that%
\begin{equation}
\chi _{\mu }(-v_{n}^{-})+\chi _{\mu }(v_{n}^{+})=\chi _{\mu }(v_{n})\text{ \
and \ }\chi _{\mu }(-(v^{\ast }){}^{-})+\chi _{\mu }((v^{\ast }){}^{+})=\chi
_{\mu }(v^{\ast }).  \label{56}
\end{equation}%
Then, in view of (\ref{48*})-(\ref{56}), for $\varphi _{2}\in \mathcal{H}%
_{+}^{1}(\Omega )$, we get%
\begin{eqnarray*}
&&\lim_{n\rightarrow +\infty }\int_{\Omega \cap \{|v_{n}|\leq \mu
\}}f_{2}(u_{n})\frac{|u_{n}|^{\alpha _{2}}}{|v_{n}+\gamma
_{n}(v_{n})|^{\beta _{2}}}\varphi _{2}\text{ }\mathrm{d}x \\
&=&\lim_{n\rightarrow +\infty }\int_{\Omega \cap \{|v_{n}|\leq \mu
\}}f_{2}(u_{n})\frac{|u_{n}|^{\alpha _{2}}}{|v_{n}+\gamma
_{n}(v_{n})|^{\beta _{2}}}\varphi _{2}\text{ }\chi _{\mu }(v_{n})\text{ }%
\mathrm{d}x \\
&\leq &\lim_{n\rightarrow +\infty }\int_{\Omega \cap \{|v_{n}|\leq \mu
\}}(\nabla v_{n}\text{\thinspace }\nabla \varphi _{2}+v_{n}\varphi _{2})\chi
_{\mu }(v_{n})\text{ }\mathrm{d}x \\
&\leq &\int_{\Omega }(\nabla v^{\ast }\text{\thinspace }\nabla \varphi
_{2}+v^{\ast }\varphi _{2})\chi _{\mu }(v^{\ast })\text{ }\mathrm{d}x.
\end{eqnarray*}%
Since $(\nabla v^{\ast }$\thinspace $\nabla \varphi _{2}+v^{\ast }\varphi
_{2})\chi _{\mu }(v^{\ast })\rightarrow 0$ a.e. in $\Omega ,$ as $\mu
\rightarrow 0,$ Lebesgue's Theorem implies that 
\begin{equation}
\lim_{\mu \rightarrow 0}\lim_{n\rightarrow +\infty }\int_{\Omega \cap
\{|v_{n}|\leq \mu \}}f_{2}(u_{n})\frac{|u_{n}|^{\alpha _{2}}}{|v_{n}+\gamma
_{n}(v_{n})|^{\beta _{2}}}\varphi _{2}\text{ }\mathrm{d}x=0.  \label{50}
\end{equation}%
On the other hand, noting that 
\begin{eqnarray*}
&&\int_{\Omega \cap \{|v_{n}|>\mu \}}f_{2}(u_{n})\frac{|u_{n}|^{\alpha _{2}}%
}{|v_{n}+\gamma _{n}(v_{n})|^{\beta _{2}}}\varphi _{2}\text{ }\mathrm{d}x \\
&=&\int_{\Omega }f_{2}(u_{n})\frac{|u_{n}|^{\alpha _{2}}}{|v_{n}+\gamma
_{n}(v_{n})|^{\beta _{2}}}\varphi _{2}\text{ }\mathbbm{1}_{\{|v_{n}|>\mu \}}%
\mathrm{d}x
\end{eqnarray*}%
and $\mathbbm{1}_{\{|v_{n}|>\mu \}}\rightarrow \mathbbm{1}_{\{|v^{\ast
}|>\mu \}}$ a.e. on $\{x\in \Omega :|v_{n}|\neq \mu \}.$ By (\ref{38}) and (%
\ref{43}), together with Lebesgue's Theorem, it follows that%
\begin{equation}
\lim_{n\rightarrow +\infty }\int_{\Omega \cap \{|v_{n}|>\mu \}}f_{2}(u_{n})%
\frac{|u_{n}|^{\alpha _{2}}}{|v_{n}+\gamma _{n}(v_{n})|^{\beta _{2}}}\varphi
_{2}\text{ }\mathrm{d}x=\int_{\Omega \cap \{|v^{\ast }|>\mu \}}f_{2}(u^{\ast
})\frac{|u^{\ast }|^{\alpha _{2}}}{|v^{\ast }|^{\beta _{2}}}\varphi _{2}%
\text{ }\mathrm{d}x.  \label{51}
\end{equation}%
From (\ref{43}) and the fact that $\mathbbm{1}_{\{|v_{n}|>\mu \}}\rightarrow %
\mathbbm{1}_{\{|v^{\ast }|>0\}}$ a.e. in $\Omega $, as $\mu \rightarrow 0,$
because the set $\{x\in \Omega :|v^{\ast }(x)|=\mu \}$ is negligible, we
infer that%
\begin{equation}
\begin{array}{l}
\lim_{\mu \rightarrow 0}\int_{\Omega \cap \{|v^{\ast }|>\mu \}}f_{2}(u^{\ast
})\frac{|u^{\ast }|^{\alpha _{2}}}{|v^{\ast }|^{\beta _{2}}}\varphi _{2}%
\text{ }\mathrm{d}x \\ 
=\int_{\Omega \cap \{|v^{\ast }|>0\}}f_{2}(u^{\ast })\frac{|u^{\ast
}|^{\alpha _{2}}}{|v^{\ast }|^{\beta _{2}}}\varphi _{2}\text{ }\mathrm{d}x
\\ 
=\int_{\Omega }f_{2}(u^{\ast })\frac{|u^{\ast }|^{\alpha _{2}}}{|v^{\ast
}|^{\beta _{2}}}\varphi _{2}\text{ }\mathrm{d}x.%
\end{array}
\label{52}
\end{equation}%
Hence, gathering (\ref{44}), (\ref{50}) and (\ref{52}) together we deduce
that (\ref{42}) is fulfilled for all $\varphi _{2}\in \mathcal{H}%
_{+}^{1}(\Omega )$.

Finally, writing $\varphi _{2}=\varphi _{2}^{+}-\varphi _{2}^{-}$ for $%
\varphi _{2}\in \mathcal{H}^{1}(\Omega )$ and bearing in mind the linearity
property of (\ref{42}) in $\varphi _{2}$, we conclude that (\ref{42}) holds
for every $\varphi _{2}\in \mathcal{H}^{1}(\Omega )$. Consequently, on
account of (\ref{40})-(\ref{42}), we may pass to the limit in (\ref{55}) to
conclude that $(u^{\ast },v^{\ast })\in \mathcal{H}^{1}(\Omega )\times 
\mathcal{H}^{1}(\Omega )$ is a solution of problem $(\mathrm{P})$ within $[-%
\underline{u},\underline{u}]\times \lbrack -\underline{v},\underline{v}]$.
Property (\ref{11}) in Theorem \ref{T1} together with Remark \ref{R1} force
that $(u^{\ast },v^{\ast })$ is nodal in the sens that the components $%
u^{\ast }$ and $v^{\ast }$ are nontrivial and at least are not of the same
constant sign.

Assume that $u^{\ast }<0<v^{\ast }$. Test the first equation in $(\mathrm{P}%
) $ by $-(u^{\ast })^{-}$ we get 
\begin{equation*}
\int_{\Omega }|\nabla (u^{\ast })^{-}|^{2}+|(u^{\ast })^{-}|^{2})\text{ }%
\mathrm{d}x=-\int_{\Omega }(f_{1}(v^{\ast })|(u^{\ast })^{-}|^{\alpha
_{1}}+\rho )(u^{\ast })^{-}\mathrm{d}x<0,
\end{equation*}%
which forces $(u^{\ast })^{-}=0$, a contradiction. So assume $v^{\ast
}<0<u^{\ast }$. Test the second equation in $(\mathrm{P})$ by $-(v^{\ast
}){}^{-}$it follows that 
\begin{equation*}
\int_{\Omega }|\nabla (v^{\ast })^{-}|^{2}+|(v^{\ast })^{-}|^{2})\text{ }%
\mathrm{d}x=-\int_{\Omega }f_{2}(u^{\ast })\frac{|u^{\ast }|^{\alpha _{2}}}{%
|(v^{\ast })^{-}|^{\beta _{2}}}(v^{\ast })^{-}\mathrm{d}x<0.
\end{equation*}%
Hence, $(v^{\ast })^{-}=0$, a contradiction. Consequently, $u^{\ast }$ and $%
v^{\ast }$ cannot be of opposite constant sign. However, considering Theorem %
\ref{T1} we can conclude that $u^{\ast }$ and $v^{\ast }$ must change sign
simultaneously and therefore, $u^{\ast }v^{\ast }\geq 0$ in $\Omega $. This
completes the proof.

\bigskip

\textbf{Funding.} This research received no specific grant from any funding
agency in the public, commercial, or not-for-profit sectors.

\bigskip

\textbf{Conflict of Interest Statement.} The author has no competing
interests to declare that are relevant to the content of this article.

\bigskip

\textbf{Data Availability Statement.} No data sets were generated or
analyzed during the current study.

\end{document}